\documentclass[11pt,reqno]{amsart}
\usepackage{amsmath,amsthm,amssymb,enumerate}
\usepackage{hyperref}
\usepackage{mathrsfs}
\usepackage[english]{babel}

\usepackage[pdftex]{color}
\usepackage[usenames,dvipsnames]{xcolor}

\newcommand{\mysection}[1]{\section{#1}
      \setcounter{equation}{0}}
\usepackage{color}

\newtheorem{theorem}{Theorem}[section]
\newtheorem{lemma}[theorem]{Lemma}

\theoremstyle{definition}
\newtheorem{assumption}{Assumption}[section]
\newtheorem{definition}{Definition}[section]

\theoremstyle{remark}
\newtheorem{remark}{Remark}[section]

\newcommand\bR{\mathbb{R}}

\newcommand\bB{\mathbb{B}}

\newcommand{\B}{\mathscr{B}}

\newcommand{\<}{\langle}
\renewcommand{\>}{\rangle}
\begin{document}

\title{A Comparison Principle 
for Stochastic Integro-Differential Equations }


\author{Konstantinos Dareiotis       \and
       Istv\'an Gy\"ongy  
}


\maketitle

\begin{abstract}
A comparison principle for stochastic 
integro-differential equations 
driven by L\'evy processes is proved. 
This result is obtained via an extension of an It\^o formula, proved by N.V. Krylov, for 
the square of the  norm of the positive part of $L_2-$valued,  continuous semimartingales,  to the case of discontinuous semimartingales.
\end{abstract}

\section{Introduction}
\label{intro}
Our goal is to prove a comparison principle 
for  stochastic integro-differential equations (SIDEs) 
driven by L\'evy processes. 
For this, first we present an It\^o's formula for the square of the $L_2$-norm 
of the positive part of (possibly) discontinuous semimartingales with values in  
$L_2$-spaces. Our formula extends an It\^o formula from \cite{K} proved  
for continuous semimartingales. In \cite{K} It\^o's formulas for the square 
of $L_2$-norm of certain convex functions $r(u)$ of continuous semimartingales 
$u=u_t$ with values in $L_2$-spaces are obtained, and the important special case, 
$r(u)=(u)^+=\max(u,0)$, is then applied to prove a maximum principle for linear stochastic partial differential equations (SPDEs). The present paper is organized as follows. In Section 2 we formulate and prove our It\^o formula. The main results concerning comparison theorems are stated in Section 3.  We also give an existence and uniqueness result as a simple consequence of a theorem on stochastic evolution equations from \cite{G3}. For recent results concerning the solvability of SPDEs driven by L\'evy processes we refer to \cite{Kim}. In Section 4 we give some tools that will be needed in order to prove the main theorems in Section 5. For notions and results in SPDEs we refer to \cite{Ro}.

Comparison principles are  powerful tools and 
play important role in PDE theory. Comparison 
theorems for SPDEs are known in  various generalities in the literature. 
To the best of our knowledge, the first results on comparison of solutions of SPDEs appear in 
 \cite{KO} and \cite{DP}. Recent results appear in  \cite{K}, \cite{M}, \cite{DMS} and \cite{DMZ}. In \cite{DMS} and \cite{M} quasi linear SPDEs, and in \cite{DMZ}  
quasi-linear SPDEs with obstacle are considered, and the $p$-th moments of the positive part of the supremum norm of the solutions are also estimated. 
In the above publications, SPDEs driven by 
Wiener processes, or cylindrical Wiener processes are considered. Our main result, 
Theorems \ref{Maintheorem} and \ref{Maintheorem2}, are comparison theorems for two classes of quasilinear SIDEs, linear versions of which, arise in non-linear filtering. For the theory of non-linear filtering of processes with jumps we refer to \cite{Gr1} and \cite{Gr2}.
We will apply our result to investigate the solvability of a class of SPDEs driven by L\'evy 
processes in another paper.

In conclusion we introduce some basic notation of the paper.  Let  $(\Omega, \mathscr{F}, P)$ be 
a  probability space  
equipped with a  
right-continuous filtration $(\mathscr{F}_t)_{t\geq0}$, 
such that $\mathscr{F}_0$ contains all 
$P$-zero sets.  
We consider  a $\sigma$-finite measure space,  
$(Z, \mathcal{Z}, \nu)$ and a  quasi left-continuous, adapted  point 
process $(p_t)_{t\in[0,T]}$ in $Z$, for a finite $T>0$.
 Let $N(dt,dz)$  be the 
random measure on $[0,T] \times Z$, corresponding to the point processes $p$.
We assume that its compensator
is  $dt \nu(dz)$ and we use the notation 
$$
\tilde{N}(dt,dz) = N(dt,dz) -dt\nu(dz).
$$
We also consider a sequence of 
independent 
real valued  $\mathscr{F}_t$-Wiener processes $\{w^k_t\}_{k=1}^{\infty}$. 

 If $X$ is a 
topological space then $\B (X)$ is  the 
Borel $\sigma$-algebra on $X$. The notation  $\mathscr{P}$ 
is used for the predictable $\sigma$-algebra on $\Omega \times [0,T]$. 
If $X$ is a normed linear space then $|x |_X $ denotes the norm of 
$x\in X$, $X^*$ is the dual of $X$, and  
$\langle x, x^* \rangle$ denotes the action of 
$x^* \in X^*$ on $x\in X$.  The notation $Q$  stands for the whole space $\bR^d$ or for a bounded Lipschitz domain in $\bR^d$. We write
$$
D_iu:=\frac{\partial u}{\partial x_i}, \ D_{ij}u:= \frac{\partial^2 u}{\partial x_i \partial x_j}, \ \text{for  $i,j=1,...,d$,}
$$
for the first and second order partial derivatives of a function $u$ defined on $Q$.
 As usual we denote by $W^k_p(Q)$ the space of functions 
$u \in L_p(Q)$, whose generalized derivatives up to order $k$
lie in $L_p(Q)$. We set $H^1(Q):=W^1_2(Q)$ and we  write 
$H^1_0(Q)$ for the closure of $C^\infty_c(Q)$ in $H^1(Q)$ under the norm
$$ 
|u|_{H^1} = \Big( \sum_{i=1}^d |D_iu|_{L_2}+|u|_{L_2} \Big)^{1/2}.
$$
We will use the notation $H^{-1}(Q)$ for the dual of $H^1_0(Q)$.
Finally, we note that unless otherwise indicated, the summation convention is used with respect to repeated integer-valued indices throughout the paper.

\mysection{It\^o's formula for the square of the norm of the positive part}   \label{section 1}

We are interested in  an It\^o's  formula for $|u_t^+|^2_{L_2(Q)}$,
where $u_t$ is an $H^{-1}(Q)$-valued semimartingale 
taking values in $H^1_0(Q)$ for 
$dP\times dt$ almost every $(\omega,t)\in\Omega\times[0,T]$.  Our approach to obtain it is similar to that in \cite{K}.
To state the  formula we set 
 $$
 V:= H^1_0(Q),\quad H:= L_2(Q),\quad V^*:= H^{-1}(Q), 
 $$
 and  we consider  the following processes
 $$
 v:\Omega \times [0,T] \to V, \ v^*: \Omega \times [0,T] \to V^*, \ h^k : 
 \Omega \times [0,T] \to H ,
 $$
 $$
 K : \Omega \times [0,T] \times Z \to H,
 $$
 for integers $k\geq1$, 
 where  $v$ is progressively measurable,
  $v^*$ and  
 $h^k$ are  $\mathscr{F}_t$-adapted, measurable in $(\omega,t)$, and 
 $K$ is $\mathscr{P} \times \mathcal{Z} $ measurable. 
 We consider also $\psi$,
  an $\mathscr{F}_0$-measurable random variable in $H$. 
  
  It is easy to see that $V=H^1_0(Q)$ is continuously and densely 
  embedded into $L_2(Q)$. Hence, by identifying $H=L_2(Q)$ with its dual 
  $H^{\ast}$ by the help of the inner product $(,)$ in $L_2(Q)$, 
   we get the {\it normal triple of spaces} 
  $$
V\hookrightarrow H\equiv H^{\ast}\hookrightarrow V^{\ast},
  $$
  where $H^{\ast}\hookrightarrow V^{\ast}$
is the adjoint embedding of $V\hookrightarrow H$. We use the notation 
$\langle v^{\ast},v\rangle$ for the duality product of $v^{\ast}\in V^{\ast}$ and 
$v\in V$. Notice that $\langle v^{\ast},v\rangle=(v^{\ast},v)$ when $v^{\ast}\in H$.

 A stochastic process $v=(v_t)_{t\in[0,T]}$, taking values in a Banach space 
$\bB$, is called a $\bB$-valued strongly c\'adl\'ag  process if with probability one 
the trajectories of $v$ are continuous from the right in $t\in[0,T)$ and have limits from 
the left at every $t\in(0,T]$ in the strong topology of $\bB$, i.e., in the topology given by the norm 
in $\bB$. 

 We make the following assumption.
 \begin{assumption} $ $                                                    \label{Assumption Ito's}
\begin{enumerate}
\item[(i)] Almost surely   
 \[ 
 \int_{(0,T]} \Big( |v_t|^2_V +     |v^*_t|_{V^*}^2+\sum_k |h^k_t|_H^2   
 +\int_Z  |K_t(z)|^2_H \nu (dz)\Big) dt < \infty,
 \]
\item[(ii)] for each 
$\phi \in V$ and for $dP \times dt$-almost every $(\omega,t)$,   we have
$$
(v_t,\phi) =(\psi,\phi)+ \int_{(0,t]} \langle v^*_s, \phi \rangle  ds + \int_{(0,t]} (h^k_s, \phi)dw^k_s 
$$
$$
+ \int_{(0,t]} \int_Z (K_s(z), \phi ) \tilde{N} (ds,dz).
$$

\end{enumerate} 
 \end{assumption}

 \begin{theorem}                                                                              \label{Ito's formula}
Suppose that Assumption \ref{Assumption Ito's} is satisfied. 
 Then there exists a set $\tilde{\Omega } \subset \Omega$ 
 of probability one, and an $H$-valued strongly c\'adl\'ag 
 adapted process $u_t$ 
 such that $u_t=v_t$ for $dP \times dt$-almost every 
 $(\omega,t)$. Moreover for $\omega \in \tilde{\Omega}$, 
 $t \in [0,T]$ we  have
 \begin{align}                                                                      \label{equalityforallt} 
 i)& \  u_t =\psi+ \int_{(0,t]} v^*_s  ds + \int_{(0,t]} h^k_sdw^k_s 
   + \int_{(0,t]} \int_Z K_s(z) \tilde{N} (ds,dz), \\                                                                                             
\nonumber  ii)& \, |u^+_t|^2_H=|\psi^+|^2_H 
+ 2\int_{(0,t]} \langle v^*_s , u^+_s \rangle ds +2\int_{(0,t]} (h^k_s,u^+_s)dw^k_s
\end{align}
 
 $$
 +2\int_{(0,t]} \int_Z (K_s(z), u^+_{s-}) \tilde{N} (dz,ds)
 + \int_{(0,t]}\sum_k  | I_{ u_s >0} h^k_s|^2_H ds
 $$
\[+ \int_{(0,t]} \int_Z |(u_{s-}+K_s(z))^+|^2_H-|u^+_{s-}|^2_H- 2(K_s(z), u^+_{s-})_H\  N(dz,ds).\]
\end{theorem}
To prove  Theorem \ref{Ito's formula}  we need two lemmas. 
\begin{lemma}                                                                                   \label{majorant}
Let $(X,\Sigma, \mu)$ be a measure space, and 
let $u_n ,u \in L_1(X)$such that  $u_n \to u$ in $L_1(X)$ as 
$n\to\infty$.
Then there exists a subsequence $\{u_{n(k)}\}_{k=1}^{\infty}$ and a function 
$v \in L_1(X)$ such that for all $k\geq1$ we have $|u_{n(k)}(x)|\leq v(x)$ for all 
$x\in X$, and $u_{n(k)}(x)\to u(x)$ for $\mu$-almost every $x$ as $k\to\infty$.
\end{lemma}
\begin{proof}
 There exists a subsequence $u_{n(k)}$ such that 
\[
|u_{n(k)}-u|_{L_1(X)} \leq 1/2^k\quad \text{for $k\geq1$}.
\]
Set $v(x)=|u(x)|+\sum_k |u_{n(k)}(x)-u(x)|$. Then $v$ has the desired properties. 
Moreover, $\sum_k |u_{n(k+1)}-u_{n(k)}|_{L_1(X)}<\infty$, which implies 
that $u_{n(k)}$ converges $\mu$-almost everywhere.
 \end{proof}
 The next lemma is from \cite{M}. 
 \begin{lemma}                                                                                    \label{trancating sequence}
 
 Let $\mathcal{Q}$ be a bounded Lipschitz domain in $\mathbb{R}^d$. Take $\phi_n \in C^\infty_c( \mathcal{Q})$, $n\in \mathbb{N}$, with 
\begin{enumerate}
\item[i)] $0 \leq \phi_n \leq 1  $

\item[ii)] $\phi_n=1$ on $\{ x\in \mathcal{Q}, r(x) \geq 1/n\}$ 

\item[iii)] $\phi_n=0$ on $ \{ x\in \mathcal{Q}, r(x) \leq 1/2n\}$

\item[iv)]  $|(\phi_n)_{x_i} | \leq Cn$, 
\end{enumerate}
where $C$ is a constant and $r(x)= dist(x, \partial \mathcal{Q})$. 
Then $\phi_n v \to v$ in $H^1_0( \mathcal{Q})$ 
for all $v \in H^1_0( \mathcal{Q})$, 
and for some constant $C$ we have 
\[\sup_n |\phi_n v |_{H^1_0} \leq C |v|_{H^1_0}, 
\qquad \forall v \in H^1_0( \mathcal{Q}).\]
\end{lemma}
\begin{remark}                                                                      \label{remarktrancating}
One can easily see the existence of a sequence 
$(\phi_n)_{n\in \mathbb{N}}$ satisfying 
the conditions of the previous lemma. 
Then note  that $\phi_n^2$ also satisfies i)-iv). 
Hence,  $\phi_n^2 v \to v$ in $H^1_0( \mathcal{Q})$, 
for all $v \in H^1_0( \mathcal{Q})$, 
and for some constant $C$ we have 
\[\sup_n |\phi_n^2 v |_{H^1_0} \leq C |v|_{H^1_0}, 
\qquad \forall v \in H^1_0( \mathcal{Q}).\]
\end{remark}
We introduce now the functions $\alpha_\delta(r)$, 
$\beta_\delta(r)$ and $\gamma_\delta(r)$ on 
$\mathbb{R}$, for  $\delta>0$,  given by
\begin{equation}
 a_ \delta (r)= \left\{
\begin{array}{rl}
1 & \text{if } r > \delta \\
\frac{r}{\delta} & \text{if } 0 \leq r \leq \delta \\
0 & \text{if } r < 0,
\end{array} \right. \nonumber
\end{equation}
\[ \beta _\delta (r) = \int_0^r a_\delta (s)ds,
\qquad \gamma_\delta(r)=\int_0^r \beta_\delta(s) ds.\]
For all $r \in \mathbb{R}$ we have $\alpha_\delta(r) \to I_{r >0}$, 
$ \beta_\delta(r) \to r^+$ 
and $\gamma_\delta(r) \to {(r^+)^2}/{2}$ as $\delta \to 0$. 
Also, for all $r,r_1,r_2$ and $\delta$,  
the following inequalities hold
\[|\alpha_\delta(r)|\leq 1, \ |\beta_\delta(r)| 
\leq |r|, \ |\gamma_\delta(r)| \leq \frac{r^2}{2},\]
\begin{equation}                        \label{Taylor estimate}
|\gamma_\delta(r_1+r_2)- \gamma_\delta(r_1)-\beta_\delta(r_1) r_2| \leq |r_2|^2.
\end{equation}
We are now ready to prove Theorem \ref{Ito's formula}.

\begin{proof}[Proof of Theorem \ref{Ito's formula}] 
We only prove ii) since the rest of the assertions are proved in \cite{G2}, in greater generality. 
First we prove the statement when $Q= \mathbb{R}^d$.
 We have that equality \eqref{equalityforallt} is satisfied if and only if,   almost surely, 
 for all $\varphi\in V$ and $t$ we have
  \[
  (u_t, \varphi)=(\psi,\varphi) 
  + \int_{(0,t]}\langle v^*_s, \varphi \rangle ds + \int_{(0,t]} ( h^k_s,\varphi) dw^k_s
  \]
 \begin{equation}                                                                               \label{equalityonV^*}
 +   \int_{(0,t]} \int_Z (K_s(z),\varphi) \tilde{N} (ds,dz).
 \end{equation}
Let $\phi$ be a mollifier with compact support and set $\phi_\epsilon(x):= \epsilon^{-d} \phi(x/ \epsilon)$. 
For fixed $x$, the function $\phi_\epsilon (x-\cdot)$ is in $V$, 
so we can plug it in \eqref{equalityonV^*} instead of  $\varphi$, 
to get that almost surely, for all $t\in[0,T]$ 
$$
u^\epsilon_t(x) =u_0^\epsilon(x)+ \int_{(0,t]} v^{*\epsilon}_s(x) ds+  \int_{(0,t]} h^{k\epsilon}_s(x)dw^k_s
$$
$$
+  \int_{(0,t]} \int_Z K^\epsilon_s(z,x) \tilde{N} (ds,dz),
$$  
where for $g \in V^*$ we use the notation $g^\epsilon(x):= \<g, \phi_\epsilon(x- \cdot)\>$.
Note that $u_0^{\epsilon }$ is $\mathscr{F}_0\times \mathscr{B}(\mathbb{R}^d)$ measurable. Also $u^\epsilon, v^{*\epsilon}$ and $h^{k\epsilon}$ are jointly measurable in $(t,\omega,x)$, $\mathscr{F}_t \times \mathscr{B}(\mathbb{R}^d)$ measurable for each $t$,  and $K^\epsilon$ is $\mathscr{P}\times \mathcal{Z} \times \mathscr{B}(\mathbb{R}^d)$ measurable. 
It is also easy to see that there exists a constant $C_\epsilon$, depending on $\epsilon$, such that for all $t, \omega, x,z$
\[ |u^\epsilon_t(x)|\leq C_\epsilon |u_t|_H,\ |u_0^{\epsilon}(x)| \leq C_\epsilon |u_0 |_H, |v^{*\epsilon}_t|_H \leq C_\epsilon |v^*_t|_{V^*} \]
\[|v^{*\epsilon}_t(x)| \leq C_\epsilon |v^*_t|_{V^*},\ |h^{k\epsilon}_t(x)| \leq C_\epsilon |h^k_s|_H,\]
\begin{equation}                                                                             \nonumber
 \ |K^\epsilon_t(x,z)| \leq C_ \epsilon|K_t(z)|_H.
 \end{equation}
 One can also check that  for a constant $C$, for all $\epsilon$
$$
 |u^\epsilon_t|_H\leq C|u_t|_H,\ |u_0^\epsilon|_H \leq C |u_0 |_H, \ |K^\epsilon_t(z)|_H \leq C|K_t(z)|_H     
 $$
\begin{equation}                                                                              \nonumber
 |h^{k\epsilon}_t|_H \leq C |h^k_s|_H, 
 |v^{*\epsilon}_t|_{V^*} \leq  C|v^*_t|_{V^*},\ |u^\epsilon_t|_V\leq C|u_t|_V.
 \end{equation}
Now let $ \alpha_\delta, \beta_\delta, \gamma_\delta$ 
be as before, and fix $x$. By It\^o's formula 
(see for example \cite{IW} or \cite{SR}), 
for each $x$ there exists a set 
$\Omega_x$ of full probability, 
such that for all $\omega \in \Omega_x$ and $t \in [0,T]$ we have
$$
\gamma_\delta (u^\epsilon_t(x))
= \gamma_\delta(u_0^\epsilon(x))
+\int_{(0,t]} \beta_\delta (u^\epsilon_s(x))v^{*\epsilon}_s(x)ds
$$
$$
 +\sum _k 
\int_{(0,t]} \beta_\delta (u^\epsilon_s(x))h^{k\epsilon}_s(x)dw^k_s
+\frac{1}{2}\sum_k  
 \int_{(0,t]} \alpha_\delta(u^\epsilon_s(x))|h^{k\epsilon}_s(x)|^2 ds
 $$
$$
+\int_{(0,t]} \int_Z \beta_\delta (u^\epsilon_{s-}(x))K^\epsilon_s(z,x) \tilde{N} (ds,dz)
$$
$$
+\int_{(0,t]} \int_Z \gamma_\delta (u^\epsilon_s(x)+K^\epsilon_s(z,x))
$$
\begin{equation}                                                                             \label{equalityIto's}
-\gamma_\delta (u^\epsilon_{s-}(x))
-\beta_\delta (u^\epsilon_{s-}(x))K^\epsilon_s(z,x) N(dz, ds).
\end{equation}
One can redefine the stochastic integrals such that \eqref{equalityIto's} 
holds for all $(\omega,t,x)$.  Integrating \eqref{equalityIto's} over $\bR^d$, 
taking appropriate versions of the stochastic integrals  
and using the Fubini and the stochastic Fubini theorems 
we get for each $t \in [0,T]$,
\[
\int_{\mathbb{R}^d}\gamma_\delta (u^\epsilon_t(x))\,dx
= \int_{\mathbb{R}^d}\gamma_\delta(u_0^\epsilon(x))\,dx
+\int_{(0,t]} \int_{\mathbb{R}^d}
 \beta_\delta (u^\epsilon_s(x))v^{*\epsilon}_s(x)\,dx\, ds
\]
\[ 
\int_{(0,t]} \int_{\mathbb{R}^d} 
\beta_\delta (u^\epsilon_s(x))h^{k\epsilon}_s(x)\,dx\, dw^k_s
+\frac{1}{2} \sum_k \int_{(0,t]}
 \int_{\mathbb{R}^d} \alpha_\delta(u^\epsilon_s(x))|h^{k\epsilon}_t(x)|^2\,dx\, ds
 \]
\[
+\int_{(0,t]} \int_Z \int_{\mathbb{R}^d} \beta_\delta (u^\epsilon_{s-}(x))
K^\epsilon_s(z,x) dx \tilde{N} (ds,dz)
\]
\[+\int_{(0,t]} \int_Z  \int_{\mathbb{R}^d} \gamma_\delta (u^\epsilon_{s-}(x)
+K^\epsilon_{s-}(z,x))
\]
\begin{equation}                                                                                 \label{beforelimit}                                                                         
-\gamma_\delta (u^\epsilon_{s-}(x))-\beta_\delta (u^\epsilon_{s-}(x))
K^\epsilon_s(z,x) dxN(dz, ds) \ (a.s.).
\end{equation}
 For a stochastic Fubini theorem we refer to \cite{K1}, 
 noting that 
 the Fubini theorem there can be extended easily, 
 by obvious changes in its proof, to our situation. 
 Since each term in the above equation is a c\'adl\'ag process in $t$, we see that \eqref{beforelimit} holds almost surely, for all $t \in [0,T]$.  
 We claim that for each $t\in[0,T]$ both sides of \eqref{beforelimit}  
 converges in probability as $\epsilon \to 0$ to give that  
\[\int_{\mathbb{R}^d}\gamma_\delta (u_t(x))dx= \int_{\mathbb{R}^d}\gamma_\delta(u_0(x))dx+\int_{(0,t]}  \langle  v^*_s ,\beta_\delta (u_s) \rangle ds\]
\[ \int_{(0,t]} \int_{\mathbb{R}^d} \beta_\delta (u_s(x))h^k_s(x)dx dw^k_s+\frac{1}{2} \sum_k \int_{(0,t]} \int_{\mathbb{R}^d} \alpha_\delta(u_s(x))|h^k_s(x)|^2dx ds\]
\[+\int_{(0,t]} \int_Z \int_{\mathbb{R}^d} \beta_\delta (u_{s-}(x))K_s(z,x) dx \tilde{N} (ds,dz)\]
\[+\int_{(0,t]} \int_Z  \int_{\mathbb{R}^d} \gamma_\delta (u_{s-}(x)+K_s(z,x))\]
\begin{equation}                                                                               \label{afterlimit}
-\gamma_\delta (u_{s-}(x))-\beta_\delta (u_{s-}(x))K_s(z,x) dx N(dz, ds).
\end{equation}
holds almost surely for each $t\in[0,T]$.
We are going to show that 
each term in \eqref{beforelimit} 
converges in probability
to the corresponding one in \eqref{afterlimit}. 
 Since for any sequence $\epsilon_k\downarrow0$ we have 
 $u^{\epsilon_k}_t\to u_t$ in $L_2(\bR^d)$ for every $\omega\in\Omega$, 
 by the equality $a^2-b^2=(a-b)(a+b)$  
we have $(u^{\epsilon_{k}}_t)^2 \to u^2_t$ in $L_1(\bR^d)$. Thus for every 
$\omega\in\Omega$ 
by Lemma \ref{majorant}  there exist $g\in L_1(\bR^d)$ and a  subsequence,  
denoted  again by $\epsilon_{k}$, 
such that for all $k\geq1$
\[
 |\gamma_\delta (u^{\epsilon_k}_t(x))| 
 \leq \frac{(u^{\epsilon_k}_t(x))^2}{2}\leq \frac{g(x)}{2} \ \text{for all $x$}.
 \]
Since
$\gamma_\delta (u^{\epsilon_k}_t(x))
\to \gamma_\delta (u_t(x))$ 
for almost every $x$, as $k\to\infty$, by Lebesgue's theorem on dominated convergence we obtain
$$
\int_{\mathbb{R}^d}\gamma_\delta (u^{\epsilon_k}_t(x))dx
\to \int_{\mathbb{R}^d}\gamma_\delta (u_t(x))dx \quad\text{as $k\to\infty$}. 
$$
Thus, for  $\epsilon\downarrow0$ the left-hand side 
of \eqref{beforelimit} converges  to the left-hand side 
of \eqref{afterlimit} for every $\omega\in\Omega$, and hence also in probability, for each $t\in[0,T]$.
To see the convergence of the second term 
in the right-hand side of \eqref{beforelimit} we 
fix $(s, \omega)$ such that $u_s \in V$. 
Then it is straightforward to check that 
\[
|\beta_\delta (u^\epsilon_s) - \beta_\delta (u_s)|_V \to 0,\ \text{as} \ \epsilon \to 0.
\]
Taking into account the well-known fact that 
there exist $f^0_s$ and $f^{i}_s\in L^2(\bR^d)$ for $i=1,...,d$ such that 
$$
v^{*}_s=f^0_s+D_if^i_s, 
$$
we have 
$$
v^{*\epsilon}_s= f^{0\epsilon}_s+ D_i f^{i\epsilon}_s,  
$$
which gives
\[
|v^*_s-v^{*\epsilon}_s|_{V^*} \leq \sum_{i=0}^d|f^{i\epsilon}_s-f^i_s|_H\to 0, \ \text{as} \ \epsilon \to 0.
\]
Hence we conclude 
\[
\int_{\mathbb{R}^d}\beta_\delta (u^\epsilon_s(x))v^{*\epsilon}_s(x)dx  
= \langle v^{*\epsilon}_s, \beta_\delta (u^\epsilon_s) \rangle \to \langle v^*_s, \beta_\delta (u_s) \rangle. 
\]
Notice that there is a constant $C$ 
such that 
\[
\Big|
\int_{\mathbb{R}^n}\beta_\delta (u^\epsilon_s(x))v^{*\epsilon}_s(x)dx 
\Big|
\leq C(|u_s|^2_V+|v^*_s|^2_{V^*})
\]
for all $\epsilon>0$, $\omega\in\Omega$ and  $s\in[0,T]$. 
Therefore, almost surely 
\[
\int_{(0,t]} \int_{\mathbb{R}^d}
\beta_\delta (u^\epsilon_s(x))v^{*\epsilon}_s(x)dx ds 
\to \int_{(0,t]} \langle v^*_s, \beta_\delta (u_s) \rangle ds 
\quad\text{for all $t$}. 
\]
For the sum of the stochastic integrals against the Wiener processes 
we just note that almost surely for all $s\in[0,T]$
\[
\sum_k \Big|\int_{\mathbb{R}^d} 
\beta_\delta (u^\epsilon_s(x))h^{k\epsilon}_s(x)dx
-\int_{\mathbb{R}^d} \beta_\delta (u_s(x))h^k_s(x)dx\Big|^2 \to 0
\quad \text{as $\epsilon\downarrow0$},
\]
and 
\[
\sum_k \Big|\int_{\mathbb{R}^d} 
\beta_\delta (u^\epsilon_s(x))h^{k\epsilon}_s(x)dx-\int_{\mathbb{R}^d} 
\beta_\delta (u_s(x))h^k_s(x)dx\Big|^2 \]
\[ \leq 4 \sup_{t\leq T}|u_t|^2_{L^2}\sum_k |h^k_s|^2_{L^2}
\quad \text{for all $\epsilon>0$}.
\]
 Hence almost surely 
\[
\int_{(0,T]} \sum_k \Big|\int_{\mathbb{R}^d}
 \beta_\delta (u^\epsilon_s(x))h^{k\epsilon}_s(x)dx
 -\int_{\mathbb{R}^d} \beta_\delta (u_s(x))h^k_s(x)dx\Big|^2 ds \to 0,
  \]
which implies that for $\epsilon\downarrow0$ 
$$
\int_{(0,t]} \int_{\mathbb{R}^d} 
\beta_\delta (u^\epsilon_s(x))h^{k\epsilon}_s(x)\,dx\, dw^k_s\to
\int_{(0,t]} \int_{\mathbb{R}^d} 
\beta_\delta (u_s(x))h^{k}_s(x)\,dx\, dw^k_s
$$
in probability, uniformly in $t\in[0,T]$.
Note that for each $k$ we have 
\[
 \Big| \int_{\mathbb{R}^d} 
\alpha_\delta(u^\epsilon_s(x))|h^{k\epsilon}_s(x)|^2
-\alpha_\delta(u_s(x))|h^k_s(x)|^2dx\Big|
\]
\[ 
\leq \int_{\mathbb{R}^d} |(h^{k\epsilon}_s(x) )^2          -(h^k_s(x))^2|dx \]
\[ 
+  \int_{\mathbb{R}^d} |h^k_s(x)|^2|        \alpha_\delta(u^\epsilon_s(x))- \alpha_\delta(u_s(x))|dx \to 0.
\]
for each $\omega\in\Omega$ and $s\in[0,T]$. Moreover, 
\[
\Big| \int_{\mathbb{R}^d} \alpha_\delta(u^\epsilon_s(x))
|h^{k\epsilon}_s(x)|^2dx\Big| \leq |h^k_s|^2_H, 
\] 
where the right-hand side is almost surely integrable on $[0,T]$. 
Hence the almost sure convergence 
of the fourth term in the right-hand side of 
\eqref{beforelimit} follows.  
By the inequalities in \eqref{Taylor estimate}, similar arguments 
show the convergence of the last two terms in probability. 
We conclude that for each $t\in[0,T]$ equation \eqref{afterlimit}
 holds almost surely. Since the stochastic processes in both sides of 
 \eqref{afterlimit} are c\`adl\`ag processes, equation \eqref{afterlimit} holds almost surely 
 for all $t\in[0,T]$.

Now by letting  $\delta \to 0$ in \eqref{afterlimit}, using arguments similar to the previous ones, 
and keeping in mind the inequalities \eqref{Taylor estimate}
and the fact that for all $v \in V$
\[ |\beta_\delta(v) - v^+ |_V \to 0, \ |\beta_\delta(v)|_V  \leq |v|_{V},\]
we can finish the proof of the theorem for  $Q = \mathbb{R}^d$. 

We  reduce the case of a bounded Lipschitz domain $Q$ to that 
of the whole space by using the sequence $\phi_n$ from 
Lemma \ref{trancating sequence}.  Remember that $\phi_n$ 
has compact support in $Q$. Thus for a function $\eta$ on $Q$ 
we denote  
by $\phi_n\eta$, not only the function defined on $Q$ by the multiplication of $\phi_n$ and $\eta$, but also its extension to zero outside of $Q$. 
Notice that when 
$u$ satisfies \eqref{equalityforallt} on  $Q$, then $\phi_n u$ 
satisfies 
\[\phi_nu_t =\phi_n u_0+ \int_{(0,t]} \phi_n v^*_s ds + \int_{(0,t]} \phi_n h^k_sdw^k_s\]
\[ + \int_{(0,t]} \int_Y \phi_n K_s(z) \tilde{N} (ds,dz)\]
on the whole $\mathbb{R}^d$,  
where the functional $\phi_nv^*$ is defined by 
$$ 
 \langle \phi_n v^*_s, g\rangle 
 := \langle  v^*_s, \phi_n g\rangle_{Q}
$$
 for $g \in H^1 ( \mathbb{R}^d)$. 
 The notation $\langle \cdot,\cdot \rangle_{Q}$ means the 
 duality product between $H^{1}_{0}(Q)$ 
 and $H^{-1}(Q)$. Notice that 
 $\langle  v^*_s, \phi_n g\rangle_{Q}$ is well defined, 
 since the restriction of $\phi_n g$ to $Q$ 
 belongs to $H^1_0 (Q)$. 
 Then by the result  in the case of the whole space, 
 we have
 \[
 \int_Q\phi_n^2 |u^+_t|^2dx=\int_Q|\phi_n u_0^
 +|^2dx + 2\int_{(0,t]} \langle v^*_s , \phi^2_n u^+_s \rangle_Q  ds\]
 \[ 
 +2\int_{(0,t]} \int_Q\phi^2_nh^k_su^+_sdxdw^k_s
+ \int_{(0,t]} \int_Q \sum_k |\mathbb{I}_{\{ \phi_n u_s >0\}} \phi_n h^k_s|^2dx ds
\]
\[
+ \int_{(0,t]} \int_Z \int_Q2K_s(z)\phi_n^2u^+_{s-}dx \tilde{N}(ds,dz)
\]
$$
+ \int_{(0,t]} \int_Z \int_Q|\phi_n (u_{s-}+K_s(z))^+|^2-|\phi_n u^+_{s-}|^2- 2K_s(z)\phi_n^2u^+_{s-} dx N (dz,ds),
$$
since $\phi_n$ is supported in $Q$.
It is now easy to take $n\to \infty$ here 
to finish the proof of the theorem. 
We only note that for the second term on the right-hand side  
we have by Lemma \ref{trancating sequence} 
and Remark \ref{remarktrancating}  
$$
\langle v^*_s, \phi^2_n u^+_s \rangle_Q 
\to \langle v^*_s, \ u^+_s \rangle_Q 
\quad\text{for all $\omega,s$}, 
$$
and for a constant $C$,
$$
\langle v^*_s, \phi^2_n u^+_s \rangle_Q \leq C|v^*_s|_{V^*} |u_s|_V
\quad\text{for all $n$}. 
$$ 
\end{proof}
\mysection{Comparison Theorems}

In this section we present our comparison theorems for  two types of equations.  Together with the space  $(Z, \mathcal{Z})$, we  consider another measurable space $(F,\mathfrak{F})$, a quasi left-continuous, adapted  point 
process $(\bar{p}_t)_{t\in[0,T]}$ in $F$, and two $\sigma-$finite measures $\pi^{(1)}, \ \pi^{(2)}$ on $F$. Let $M(dt,d\zeta)$ be the corresponding random measure on $[0,T] \times F$. We assume that its compensator is $dt \pi^{(2)}(d\zeta)$ and we write
$$
\tilde{M}(dt,d\zeta) = M(dt,d\zeta) -dt\pi^{(2)}(d\zeta).
$$
First we consider the equation
\begin{align}                                                                               
\label{SPDE1}
\nonumber du_t(x)=& \{ L_tu_t(x)+f_t(x,u_t(x), \nabla u_t(x))  \}\,dt\\          \nonumber
&+G^k_t(u)(x) dw^k_t+\int_Z g_t(x,z,u_{t-}(x))\tilde{N}(dt,dz),  \\
\end{align}
for $(t,x)\in[0,T]\times Q$, 
with initial condition
\begin{equation}                                        \label{initial 1}
u_0(x)=\psi(x),\quad x\in Q,  
\end{equation}
where
\begin{align}
\nonumber
 L_tu(x)&= D_j( a^{ij}_t(x)D_iu(x) )+\mathcal{I}^{(1)}_tu_t(x),  \\                  \nonumber
\mathcal{I}^{(1)}_tu(x)&=\int_F [u(x+c_t(x,\zeta))- u(x)- c_t(x,\zeta) \cdot \nabla u(x)]m_t(x,\zeta) \pi^{(1)} (d\zeta), \\
              \nonumber
G^k_t(u)(x)&=\phi^{ik}_t(x)D_iu(x) +\sigma^k_t(x,u(x)).
\end{align}

We make the following assumptions. Let $K>0$ denote a constant.

\begin{assumption}$ $                                                       \label{Assumptionforspde}

i)  The coefficients $a^{ij}$, are real-valued 
$\mathscr{P} \times \B (Q)$ measurable functions 
on $\Omega\times[0,T]\times Q$ and are  
bounded by $K$ for every $i,j=1,...,d$. 
The coefficient $\phi^i=(\phi^{ik})_{k=1}^{\infty}$ 
is an $l_2$-valued 
$\mathscr{P} \times \B (Q)$-measurable function 
on $\Omega\times[0,T]\times Q$ 
for every $i=1,2,...,d$, such that 
$$
\sum_i \sum_k |\phi^{ik}_t(x)|^2 \leq K \quad 
\text{for all $\omega$, $t$ and $x$}.
$$

ii) $f$ is a real valued 
$\mathscr{P} \times \B(Q) \times \B( \mathbb{R}) \times \B(\bR^d)$
-measurable function on 
$\Omega \times [0,T] \times Q \times \mathbb{R} \times \bR^d$,
 and $\sigma=(\sigma^k)_{k=1}^\infty$ is a 
$\mathscr{P} \times \B(Q) \times \B( \mathbb{R})$-measurable 
function on  
$\Omega \times [0,T] \times Q \times \mathbb{R}$, 
with values in $l_2$.   The function $g$ is defined on 
 $\Omega \times [0,T] \times Q \times Z \times \mathbb{R}$ 
 with values in $\bR$ and it is 
 $\mathscr{P} \times \B(Q) \times  \mathcal{Z}  \times \B( \mathbb{R})$-measurable.  
 We assume that there exists a
 predictable  process $\bar h_t$ with values in $L_2(Q)$, such that  almost surely $\bar h \in L_2([0,T]\times Q)$,  
 and for all $\omega,t,x,z,r,r'$ 
$$
 |f_t(x,r,r')|^2+\sum_k |\sigma^k_t(x,r)|^2 + \int_Z |g_t(x,z,r)|^2 \nu(dz) 
 $$
 $$
 \leq K|r|^2+K|r'|^2+|\bar h_t(x)|^2.
 $$

iii) $\psi$ is an $\mathscr{F}_0$-measurable random variable in $L_2(Q)$.

iv) There exists a constant $\varkappa >0$ 
such that for all $\omega,t,x$ and for all $ \xi =( \xi _1,...\xi _d) \in \mathbb R^d$ we have
\[ 
a^{ij}_t(x)\xi_i \xi_j  -\frac{1}{2}  \phi^{ik}_t(x)\phi^{jk}_t(x)\xi_i \xi _j 
\geq \varkappa |\xi |^2.
\]

v) For all $\omega,t,x,z,r_1,r_2$
$$ 
 \sum_k |\sigma^k_t(x,r_1)-\sigma^k_t(x,r_2)|^2\leq K|r_1-r_2|^2.
$$

\end{assumption}

\begin{assumption}               \label{continuity of f}
The function  $f_t(x,r,r')$ is continuous in $r$, for each $\omega,t,x$ and $r'$. 
\end{assumption}

\begin{assumption}             \label{monotnicity of f}
 For all $\omega,t,x,r_1,r_2,r'_1,r'_2$
$$
2(r_1-r_2)(f_t(x,r_1,r_1')-f_t(x,r_2,r'_1))
$$
$$
 + \int_Z|g_t(x,z,r_1)-g_t(x,z,r_2)|^2 \nu(dz) 
\leq K|r_1-r_2|^2,
$$
and
$$
|f_t(x,r_1,r'_1)-f_t(x,r_1,r'_2)| \leq K|r'_1-r'_2|. 
$$
\end{assumption}
\begin{assumption}                                              \label{assumptionforcomparison}
 The function $r+g_t(x,z,r)$ is non-decreasing in $r$ for all 
 $\omega,t,x,z$. 
\end{assumption}
\begin{assumption}                                                  \label{assumptions on c}
The function $c$ maps 
$\Omega \times [0,T] \times \bR^d \times F$ into $\mathbb R^d$, 
it is  
 $\mathscr{P}\times \B(\bR^d) \times \mathfrak{F}$-measurable,  
 and there exists an $ \mathfrak{F}$-measurable real function $\bar c$ on $F$
 such that 

(i) $|c_t(x,\zeta) |\leq \bar c(\zeta), \text{for all $\omega, t,x, \zeta$},$

(ii) $\int_F \bar c^2(\zeta) \pi^{(1)}(d\zeta) \leq K,$

(iii) $|c_t(x,\zeta)-c_t(y,\zeta)| \leq \bar c(\zeta) |x-y|$,  for all $\omega, t, x, y, \zeta$.

\end{assumption}

\begin{assumption}                                                             \label{assumptions on m}
The function $m$ maps 
$\Omega \times [0,T] \times \bR^d \times F$ 
into  $\bR$,  
it is 
$\mathscr{P}\times \B(\bR^d) \times \mathfrak{F}$- 
measurable, and we have

(i) $0 \leq m_t(x, \zeta) \leq K$, for all $\omega, t,x, \zeta$,

(ii)$|m_t(x,\zeta)-m_t(y,\zeta)| \leq K|x-y|, \text{ for all $\omega, t,x,y, \zeta$}$.
\end{assumption}

\begin{assumption}                      \label{smoothness of c}
The functions $c^l_t(x, \zeta)$, $l=1,...,d$, 
are twice continuously differentiable in $x$, for each $\omega, t, \zeta$, and

(i) $|D_i c^l_t(x, \zeta )|\leq K$, $|D_{ij} c^l_t(x,\zeta)| \leq K$,
for all $i,j,l=1,...,d$,

(ii) $K^{-1} \leq |det(\mathbb{I}+ \theta \nabla c_t(x,\zeta) )|$ 
\newline
for all $\omega, t, x, \zeta$ and $ \theta \in [0,1]$, where $\mathbb{I}$ denotes the identity matrix.
\end{assumption}

\begin{remark} 
Denote by $T_{\theta,t,\zeta}$ the mapping $x \mapsto x+\theta c_t(x,\zeta)$, for fixed $\omega,t,\theta$ and $\zeta$. By virtue of the inverse function theorem, it follows from (ii) of Assumption  \ref{smoothness of c}  that $T_{\theta,t,\zeta}$ is a $local$ diffeomorphism. In addition, by the first inequality in (i) and by (ii) of Assumption \ref{smoothness of c},   there exists a constant $\gamma >0$, such that the norm of the matrix $(\mathbb{I}+ \theta \nabla c_t(x,\zeta) )^{-1}$ is uniformly bounded by $\gamma$. Hence, by Hadamard's theorem (see, eg, Theorem 5.1.5 in \cite{Be}), $T_{\theta,t,\zeta}$ is a $global$ diffeomorphism,  for fixed $\omega,t,\theta$ and $\zeta$. We denote by $J_{\theta,t,\zeta}$ the inverse of $T_{\theta,t,\zeta}$. Notice that for fixed $\omega,t,\theta$,$\zeta$ and for all $j=1,...,d$, the functions $J^j_{\theta,t,\zeta}(x)$ are twice continuously in $x$, and their first and second order derivatives are uniformly bounded.

\end{remark}

\begin{remark}
Under Assumptions \ref{assumptions on c} through 
\ref{smoothness of c}, $\mathcal{I}^{(1)}_t$ is 
a bounded linear operator from $H^1_0(Q)$ into $H^{-1}(Q)$ for fixed $(\omega,t)$, 
and for all $u,v \in H^1_0(Q)$ the process 
$\<\mathcal{I}^{(1)}_tu, v \>$ is predictable. 
To see this, consider first the case 
  $Q=\bR^d$.
For $u \in C^\infty_c(\bR^d)$(even for $u\in W^2_2(\bR^d)$) 
one can easily see that 
$\mathcal{I}^{(1)}_tu(x)$
is a function in $L_2(\bR^d)$. Then for $v \in C_c^\infty(\bR^d)$ we have by Taylor's formula
\begin{align}                                                                                          \label{representation_of_operator3}
        \nonumber    & 
         (\mathcal{I}^{(1)}_tu,v)= \\
         \nonumber &\int_0^1 (1-\theta) 
\int_F \int_{\bR^d} 
D_{ki}u(T_{\theta, t,\zeta})c^i_t( x,\zeta)c^k_t( x,\zeta) m_t(x,\zeta) v(x)\,dx \pi^{(1)}(d\zeta)\,d\theta\\ 
&=\int_0^1 (\theta-1) 
\int_F \int_{\bR^d} 
D_iu(x+\theta c_t(x,\zeta)) D_j
( q^{ij}_t(x,\zeta, \theta) v(x))\,dx \pi^{(1)}(d\zeta)\,d\theta,
\end{align}
where the last equality is obtained by  integration by parts, 
and $q^{ij}$ is given by 
$$
q^{ij}_t(x,\zeta, \theta):= \sum_{l=1}^d c^l_t(x,\zeta)c^i_t( x,\zeta) m_t(x,\zeta)  D_lJ^j_{\theta, t,\zeta}(T_{\theta,t,\zeta}(x)).
$$
Due to Assumptions 
\ref{assumptions on c} through \ref{smoothness of c}
for a constant $N=N(d,K)$, 
\begin{equation}                           \nonumber
         ( \mathcal{I}^{(1)}_tu,v)\leq  N |u|_{H^1(\bR^d)}|v|_{H^1(\bR^d)}, 
\end{equation}
which shows
 that $\mathcal{I}^{(1)}_t$ extends uniquely 
 to a bounded linear  operator from $H^1$ to $H^{-1}$, 
 and the duality product $\langle {\mathcal I}^{(1)}_{t}u,v\rangle $ 
 is given by the right-hand side of
\eqref{representation_of_operator3}. 
In case $Q$ is a bounded Lipschitz domain, 
one can define the action of $\mathcal{I}^{(1)}_tu$ on $v\in H^1_0(Q)$ again by \eqref{representation_of_operator3}, where $u$ and $v$ this time are extended to zero outside of $Q$. 
For further study of these operators we refer to \cite{GM}.
\end{remark}

\begin{definition}
A strongly c\`adl\`ag adapted process $u$ with values in $L_2(Q)$ is called a solution of the problem \eqref{SPDE1}-\eqref{initial 1}
 if

i) $u_t \in H^1_0( Q )$ for $dP \times dt$ almost 
every $(\omega,t) \in \Omega \times [0,T]$, 

ii) $\int_{(0,T]} |u_t|^2_{H_0^1} dt < \infty$ $(a.s.)$, 

iii) for all $\varphi \in H^1_0(Q)$ we have almost surely

$$
 (u_t,\varphi)=(\psi,\varphi)
 + \int_{(0,t]} \{- (a^{ij}_s D_iu_s,D_j\varphi)
+(f_s(u_s, \nabla u_s),\varphi)
+\< \mathcal{I}^{(1)}_su_s,\varphi\> \} ds
$$
$$
+  \int_{(0,t]}\{ (\phi^{ik}_s D_iu_s ,\varphi)
+(\sigma^k_s(u_s), \varphi)\}dw^k_s
+ \int_{(0,t]} \int_Z (g_s(z,u_{s-}),\varphi) \tilde{N}(dz,ds)
$$
for all $t \in [0,T]$, where $(\cdot, \cdot)$ is the inner product in $L_2(Q)$. 

\end{definition}

\begin{theorem}         \label{existence equation 1}
Let Assumptions \ref{Assumptionforspde} through \ref{monotnicity of f} and \ref{assumptions on c} through \ref{smoothness of c} hold. Then there exists a unique solution of the problem \eqref{SPDE1}-\eqref{initial 1}.
\end{theorem}

After some preliminaries we will see that Theorem \ref{existence equation 1} follows easily from Theorems 2.9 and 2.10 from \cite{G3}.

Together with  \eqref{SPDE1}-\eqref{initial 1} let us also consider the problem

\begin{align}                                                                      
\nonumber dv_t(x)=& \{ L_tv_t(x)+F_t(x,v_t(x), \nabla v_t(x)) \}\,dt\\                                           \label{SPDE1,2}
&+G^k_t(v)(x) dw^k_t+\int_Z g_t(x,z,v_{t-}(x))\tilde{N}(dt,dz), \\
                        \label{initial 1,2}
v_0(x)=&{\Psi}(x),
\end{align}
where $F$ satisfies ii) from Assumption \ref{Assumptionforspde} and $\Psi$ is an $\mathscr{F}_0$-measurable random variable in $L_2(Q)$.
\begin{theorem}                                                                 \label{Maintheorem}
Suppose that Assumptions \ref{Assumptionforspde} 
and  
\ref{assumptionforcomparison} through \ref{smoothness of c} hold. 
Let $u$ and $v$ be solutions of the problems \eqref{SPDE1}-\eqref{initial 1} and \eqref{SPDE1,2}-\eqref{initial 1,2} respectively.
Suppose that either $f$ or $F$ satisfy Assumption \ref{monotnicity of f}.
Let $f \leq F$ and $\psi \leq\Psi$. 
Then almost surely, for all $t \in [0,T]$ we have $u_t(x) \leq v_t(x)$ for almost every $x \in Q$.
\end{theorem} 
\begin{remark} 
Assumption \ref{assumptionforcomparison} cannot be 
omitted in Theorem \ref{Maintheorem}. Consider  for example 
the SDE
\[ u_t=1-\int_{(0,t]} 2 u_{s-}d\tilde{N}_s,\]
where $N_t$ is a Poisson process with intensity one.  
Let $\tau$ be the time that the first jump of $N$ occurs. 
Then $P( \tau \leq T)>0$. Since $u_t= e^{-2t}$ on $[0, \tau)$, 
 one can see that on the set $\{ \tau \leq T\}$ 
 we  have $u( \tau) = - e ^{-2\tau} < 0$. 
\end{remark}
The second equation that we will deal with is 

\begin{align}
 \label{SPDE2}
\nonumber  du_t(x)=& \{\mathcal{L}_tu_t(x)+f_t(x,u_t(x), \nabla u_t(x)) \}dt \\
&+G^k_t(u_t)(x) dw^k_t   + \int_F S_{t,\zeta}u_{t-}(x) \ \tilde{M}(ds,d\zeta)
\end{align}
for $(t,x)\in[0,T]\times \bR^d$, 
with initial condition
\begin{equation}                                                        \label{initial 2}
u_0(x)=\psi(x),\quad x\in \bR^d,  
\end{equation}
where 
\begin{align}  \nonumber
 \mathcal{L}_tu(x)
 &= L_tu(x)+\mathcal{I}^{(2)}_tu(x),\\                  \nonumber
\mathcal{I}^{(2)}_tu(x)
&=\int_F [\lambda_t(x+b_t(\zeta), \zeta )u(x+b_t(\zeta))- \lambda_t(x,\zeta)u(x)\\
&- b_t(\zeta) \cdot \nabla( \lambda_t(x,\zeta)u(x))]\pi^{(2)}(d\zeta),\\
S_{t,\zeta} u(x)&=\lambda_t(x+b_t(\zeta),\zeta)u(x+b_t(\zeta))-\lambda_t(x,\zeta)u(x) \\
\nonumber
&+(\lambda_t(x,\zeta)-1)u(x).
\end{align}
 Obviously, if we ask later  for some of the previous assumptions 
 to hold for equation \eqref{SPDE2}, we mean with $g \equiv 0$.
\begin{assumption}                                                                    \label{assumptionequation2}

The function $b$ maps 
$\Omega \times [0,T] \times F$ into $\bR^d$, it is
$\mathscr{P} \times \mathfrak{F}$-measurable, and 
there exists an $\mathfrak{F}$-measurable real function $\bar b$ on $F$, 
such that for all $\omega, t$ and $\zeta$ we have 

$$
|b_t(\zeta)| \leq \bar b(\zeta), 
 \ \ \int_F \bar b^2(\zeta)\pi^{(2)}(d\zeta) \leq K.
$$
The function $\lambda$ maps $\Omega \times [0,T] \times \bR^d \times F$ to $[0, \infty)$, is
 $\mathscr{P} \times \B(\bR^d) \times \mathfrak{F}$-measurable,
it  is twice continuously differentiable in $x$ for all $\omega, t$,$\zeta$, and we have
$$
|\lambda_t(x, \zeta)|+|\nabla\lambda_t(x,\zeta)| +|\nabla^2\lambda_t(x,\zeta)|\leq K, 
$$
$$ 
|1-\lambda_t(x,\zeta)| \leq{\bar b}(\zeta),  \ \text{for all $\omega,t,x,\zeta$.}
$$
\end{assumption}

It is easy to see that due to Assumption \ref{assumptionequation2} for every $t\in[0,T]$ and 
$\omega\in\Omega$ the mapping ${\mathcal I}^{(2)}_t$, defined in the same way as 
${\mathcal I}_t^{(1)}$, is a bounded linear operator from $H^1$ to $H^{-1}$, 
and $\langle{\mathcal I}^{(2)}\phi,\varphi\rangle$ is 
a predictable process for any $\phi,\varphi\in H^1$.

The solution of the problem  \eqref{SPDE2}-\eqref{initial 2} is understood in the same sense as  
 that of \eqref{SPDE1}-\eqref{initial 1},  and we have the following existence and uniqueness result.
 \begin{theorem}         \label{existence equation 2}
Let Assumptions \ref{Assumptionforspde} through \ref{monotnicity of f} and \ref{assumptions on c} through \ref{assumptionequation2} hold. Then there exists a unique solution of the problem \eqref{SPDE2}-\eqref{initial 2}.
\end{theorem}
We also consider the problem
\begin{align}                                                                      
\nonumber  dv_t(x)=& \{\mathcal{L}_tv_t(x)+F_t(x,v_t(x), \nabla v_t(x)) \}dt \\ \label{SPDE 2, 2}
&+G^k_t(v)(x) dw^k_t   + \int_F S_{t,\zeta}v(x) \ \tilde{M}(ds,d\zeta),\\ \label{initial 2,2}
v_0(x)=&{\Psi}(x),
\end{align}
where $F$ and $\Psi$ are as in \eqref{SPDE1,2}-\eqref{initial 1,2}.
\begin{theorem}
                                                                 \label{Maintheorem2}
Suppose that Assumptions \ref{Assumptionforspde}, 
and \ref{assumptions on c}  through \ref{assumptionequation2} hold. Let $u$ and $v$ solve \eqref{SPDE2}-\eqref{initial 2} and \eqref{SPDE 2, 2}- \eqref{initial 2,2}
respectively. Suppose that either $f$ or $F$ 
satisfy Assumption \ref{monotnicity of f}.
Let $f \leq F$ and $\psi \leq\Psi$. 
Then almost surely, for all $t \in [0,T]$ 
we have $u_t(x) \leq v_t(x)$ for almost every $x \in \bR^d$.
\end{theorem}

\section{Auxiliary Facts}
 In this section we present some lemmas that we will 
 need for the proofs of Theorems \ref{existence equation 1}  through \ref{Maintheorem2}.  
 The following is well known (see, e.g., 
 \cite{MM}, or exercise 1.3.19 in \cite{KLp}, or some more general results 
 in \cite{TW}).

\begin{lemma}                                              \label{positivepartis continuous}
Let $u\in W^1_p(Q)$. Let $u_n\in W^1_p(Q)$ such that 
$|u_n-u|_{W^1_p} \to 0$ as $n\to\infty$. 
Then we have 
$|u_n^+-u^+|_{W^1_p} \to 0$.
\end{lemma}

For the next three lemmas, we assume that  
Assumptions \ref{assumptions on c} 
 through
\ref{assumptionequation2} hold. For $u\in C^\infty_c(\bR^d)$, let us define the quantities,
$$
\varrho_t(u):=\int_{\bR^d} \int_F \left(\lambda_t\left(x+b_t\left(\zeta\right)\right) u \left(x+b_t\left(\zeta\right)\right)\right)^2
-(\lambda_t(x,\zeta)u(x))^2
$$
$$
-2b_t(z) \cdot \nabla\left(\lambda_t \left(x,\zeta \right)u\left(x \right) \right)\lambda_t\left(x, \zeta \right) u\left( x \right) \pi^{(2)} (d\zeta) dx,
$$
$$
\tilde{\varrho}_t(u):=\int_{\bR^d} \int_F [\left(\lambda_t\left(x+b_t\left(\zeta\right)\right) u \left(x+b_t\left(\zeta\right)\right)\right)^+]^2
-[(\lambda_t(x,\zeta)u(x))^+]^2
$$
$$
-2b_t(z) \cdot \nabla\left(\lambda_t \left(x,\zeta \right)u\left(x \right) \right)\lambda_t\left(x, \zeta \right) u^+\left( x \right) \pi^{(2)} (d\zeta) dx.
$$
\begin{lemma}                            \label{estimate I(u^2)}
 For any $u \in C^\infty_c(\bR^d)$, 
$\omega\in\Omega$, $t\in[0,T]$
and $ \varepsilon >0 $ we have
\begin{equation}                            \label{estimate u^2}
\int_{\bR^d} \mathcal{I}^{(1)}_t u^2 (x) \ dx 
\leq \varepsilon |u|_{H^1(\bR^d)}^2+N(\varepsilon)|u|^2_{L_2(\bR^d)},
\end{equation}
\begin{equation}                              \label{estimate u^+^2}
\int_{\bR^d} \mathcal{I}^{(1)}_t(u^+)^2 (x) \ dx 
\leq \varepsilon |u^+|_{H^1(\bR^d)}^2+N(\varepsilon)|u^+|^2_{L_2(\bR^d)},
\end{equation}
\begin{equation}                       \label{estimate varrho}
\varrho_t(u) \leq \varepsilon |u|_{H^1(\bR^d)}^2+N(\varepsilon)|u|^2_{L_2(\bR^d)},
\end{equation}
\begin{equation}                        \label{estimate varrho tilde}
\tilde{\varrho}_t(u) \leq \varepsilon |u^+|_{H^1(\bR^d)}^2+N(\varepsilon)|u^+|^2_{L_2(\bR^d)},
\end{equation}
where the constant $N(\varepsilon)$ depends 
only on $\varepsilon$, $K$ and $d$.
\end{lemma}
\begin{proof}
We prove \eqref{estimate u^+^2}. 
For $\delta>0$ let ${\mathcal I}^{(1\delta)}$ and 
$\bar{\mathcal I}^{(1\delta)}$ denote the operators defined 
as ${\mathcal I}^{(1)}$ with 
$F$ replaced by 
$$
F_{\delta}=\{\xi\in F:\bar c(\xi)<\delta\}
$$ 
and by $F^c_{\delta}=F\setminus F_{\delta}$, 
respectively.  
Then clearly, 
\begin{equation}                                                                                      \label{I(u^2)}
\int_{\bR^d} \mathcal{I}^{(1)}_t(u^+)^2(x)dx
= \int_{\bR^d} \mathcal{I}^{(1\delta)}_t(u^+)^2(x)\,dx
+\int_{\bR^d} \bar{\mathcal{I}}^{(1\delta)}_t(u^+)^2(x)\,dx.
\end{equation}
The first term on the right-hand side is equal to 
$$
\int_0^1 (1-\theta) \int_{F_\delta} \int_{\bR^d} D_{ij}(u^+)^2(x+\theta 
c_t(x,\zeta))
$$
$$
\times c^i_t(x,\zeta)c^j_t(x,\zeta) m_t(x,\zeta)  dx \pi^{(1)}(d\zeta) d\theta
$$
$$
=E_1(t,\delta)+ E_2(t,\delta),
$$
where
$$
E_1(t,\delta)=\int_0^1 (1-\theta) \int_{F_\delta} \int_{\bR^d}2D_iu^+(x+\theta 
c_t(x,\zeta))D_ju^+(x+\theta 
c_t(x,\zeta))
$$
$$
\times  c^i_t(x,\zeta)c^j_t(x,\zeta) m_t(x,\zeta)  dx \pi^{(1)}(d\zeta) d\theta, 
$$
$$
E_2(t,\delta)=\int_0^1 (1-\theta) \int_{F_\delta} \int_{\bR^d}2u^+(x+\theta 
c_t(x,\zeta))D_{ij}u(x+\theta 
c_t(x,\zeta))
$$
$$
\times  c^i_t(x,\zeta)c^j_t(x,\zeta) m_t(x,\zeta)  dx \pi^{(1)}(d\zeta) d\theta.
$$
Using Assumptions \ref{assumptions on c}, \ref{assumptions on m} and \ref{smoothness of c}, 
we see after a change of variables that 
$$
|E_1(t,\delta)| \leq C(\delta) C|u^+|^2_{H^1(\bR^d)},
$$
where $C(\delta)= \int_{F_\delta} \bar c^2(\zeta) \pi(dz)$ 
and $C$ is a constant depending only on $K$ and $d$. For $E_2$ we have
$$
E_2(t,\delta)=\int_0^1 (1-\theta) \int_{F_\delta} \int_{\bR^d}2D_j(D_iu(x+\theta 
c_t(x,\zeta)) 
$$
$$
\times q^{ij}_t(x,\zeta, \theta) u^+(x+\theta 
c_t(x,\zeta))\, dx \pi^{(1)}(d\zeta)\, d\theta.
$$
By integration by parts and using the Assumptions \ref{assumptions on c}, \ref{assumptions on m} and \ref{smoothness of c} again we see that 
$$
|E_2(t,\delta)| \leq C(\delta) C|u^+|^2_{H^1(\bR^d)}.
$$
For the second term in \eqref{I(u^2)}  
by Young's inequality and Assumptions \ref{assumptions on m}, 
\ref{assumptions on c}, we have
$$
\int_{\bR^d} \bar{\mathcal{I}}^{(1\delta)}_t(u^+)^2(x) \ dx
 \leq \gamma |u|^2_{H^1(\bR^d)}+C(\gamma)|u|^2_{L_2(\bR^d)},
 $$
 for all $\gamma > 0$, where $C(\gamma)$ 
 depends only on $\gamma$ and $K$.
 Putting these estimates together  
 and choosing $\delta$ and $\gamma$ sufficiently small, we finish the proof  of  \eqref{estimate u^+^2}. One can repeat the same calculation with $c$ replaced by $b$, $m=1$ and $\lambda u$ in place of $u$ to get \eqref{estimate varrho tilde}. Also \eqref{estimate I(u^2)} and \eqref{estimate varrho} can be proved in the same way.
\end{proof}
\begin{lemma}                            \label{estimate positive part}
For any $u \in H^1_0(Q)$, 
$\omega\in\Omega$, $t\in [0,T]$  and $ \varepsilon >0 $ we have
\begin{equation}                      \label{equ estimate u}
2\<\mathcal{I}^{(1)}_tu,u
 \> \leq \varepsilon |u|_{H^1_0(Q)}^2+N(\varepsilon)|u|^2_{L_2(Q)},
\end{equation}
\begin{equation}                      \label{equ estimate positive part}
2\<\mathcal{I}^{(1)}_tu,u^+
 \> \leq \varepsilon |u^
 +|_{H^1_0(Q)}^2+N(\varepsilon)|u^+|^2_{L_2(Q)},
\end{equation}
 where the constant $N(\varepsilon)$ 
depends only on $\varepsilon$ and $K$ and $d$.
\end{lemma}
\begin{proof}
We prove \eqref{equ estimate positive part}. It suffices to prove it  for $Q=\bR^d$.
Due to Lemma \ref{positivepartis continuous} 
and the continuity of the operator 
$\mathcal{I}^{(1)}_t:H^1\to H^{-1}$, we may and will also assume that $u\in C^\infty_c(\bR^d)$. 
Notice that for any $\alpha,\beta\in\bR$ 
\begin{equation}                          \label{arithmetic inequality}
2(\beta-\alpha)\alpha^+\leq(\beta^+)^2-(\alpha^+)^2-(\beta^+-\alpha^+)^2
\leq (\beta^+)^2-(\alpha^+)^2.
\end{equation}
Consequently, for any $\alpha,\beta,\gamma\in\bR$ 
$$
2(\beta-\alpha-\gamma)\alpha^+\leq(\beta^+)^2-(\alpha^+)^2-2\gamma\alpha^+. 
$$
Using this with 
$\alpha=u(x)$, $\beta=u(x+c_t(x,\zeta))$ and $\gamma=c_t(x,\zeta)\nabla u(x)$, 
and taking into account that 
$
2\nabla uu^+=\nabla (u^+)^2, 
$
we can easily see that 
$$
2\<\mathcal{I}^{(1)}_tu,u^+ \>=2(\mathcal{I}^{(1)}_tu,u^+ )
\leq  \int_{\bR^d}\mathcal{I}^{(1)}_t(u^+)^2 (x) \ dx. 
$$
Hence \eqref{equ estimate positive part}  
follows from  Lemma \ref{estimate I(u^2)}. 
One can prove \eqref{equ estimate u} 
in the same way, by using the inequality 
$2(\beta-\alpha)\alpha \leq \beta^2- \alpha^2$, 
instead of \eqref{arithmetic inequality}. 
\end{proof}

For $u \in H^1(\bR^d)$ we set 
\begin{align}
\nonumber
&\mu_t(u):= 
 \int_F \int_{\bR^d} [(\lambda_t(x+b_t(\zeta),\zeta)u(x+b_t(\zeta))]^2-[u(x)]^2 \\  \nonumber
   &-2u(x)[\lambda_t(x+b_t(\zeta),\zeta)u(x+b_t(\zeta))-u(x)]dx \pi^{(2)}(d\zeta),
\end{align}
\begin{equation}                              
\rho_t(u):=2\<\mathcal{I}^{(2)}_tu,u \>+ \mu_t(u),
\end{equation}
\begin{align}
\nonumber
&\tilde{\mu}_t(u):= 
 \int_F \int_{\bR^d} [(\lambda_t(x+b_t(\zeta),\zeta)u(x+b_t(\zeta)))^+]^2-[u^+(x)]^2 \\  \nonumber
   &-2u^+(x)[\lambda_t(x+b_t(\zeta),\zeta)u(x+b_t(\zeta))-u(x)]dx \pi^{(2)}(d\zeta),
\end{align}
\begin{equation}                              \label{rho}
\tilde{\rho}_t(u):=2\<\mathcal{I}^{(2)}_tu,u^+ \>+  \tilde{\mu}_t(u).
\end{equation}
Using the simple inequality $|[(x+y)^+]^2-[x^+]^2-2x^+y| \leq   
2|y|^2$, and Assumption \ref{assumptionequation2} one can see that $\tilde{\mu}_t(u)$ is continuous in $u \in H^1(\bR^d)$.   
It can be shown similarly that  $\mu_t(u)$ is continuous in $u\in H^1(\bR^d)$.

\begin{lemma}                            \label{estimate positive part eq 2}
For any $u \in H^1(\bR^d)$, $(\omega,t) \in \Omega \times \bR^+$  and $ \varepsilon >0 $ we have
         \begin{equation}              \label{estimate integrals equation 2 without +}
\rho_t(u)   \leq \varepsilon |u|_{H^1(\bR^d)}^2+N(\varepsilon)|u|^2_{L_2(\bR^d)},
\end{equation}    
 \begin{equation}              \label{estimate integrals equation 2}
\tilde{\rho}_t(u)   \leq \varepsilon |u^+|_{H^1(\bR^d)}^2+N(\varepsilon)|u^+|^2_{L_2(\bR^d)}.
\end{equation}
\end{lemma}
\begin{proof}
Since \eqref{estimate integrals equation 2 without +} 
can be shown in the same way as \eqref{estimate integrals equation 2}, we only prove the latter one. Clearly it suffices to prove it 
for $u\in C^\infty_c(\bR^d)$. A simple calculation shows that 
$$
\tilde{\rho}_t(u)=
\tilde{\varrho}_t(u)+ \int_F\int_{\bR^d} (\lambda_t(x,\zeta)-1)^2[ u^+(x)]^2 dx\pi^{(2)}(d\zeta)
$$
$$
+\int_F\int_{\bR^d}2b_t(\zeta) \cdot  \nabla(u(x)\lambda_t(x,\zeta))  u^+(x) (\lambda_t(x,\zeta)-1) dx\pi^{(2)}(d\zeta)
$$

By Young's inequality, Assumption \ref{assumptionequation2} and \eqref{estimate varrho tilde} we get that
$$
\tilde{\rho}_t(u)   \leq \varepsilon |u^+|_{H^1(\bR^d)}^2+N(\varepsilon)|u^+|^2_{L_2(\bR^d)}.
$$
\end{proof}
\begin{lemma}         \label{f and g}
Let Assumption \ref{monotnicity of f} hold. Then for all $(\omega ,t) \in \Omega \times [0,T]$, $u\in H^1_0(Q)$ and $\varepsilon > 0 $ we have
$$
2(f_t(u,\nabla u)-f_t(v,\nabla v), u-v)+\int_Z|g_t(z,u)-g_t(z,v)|^2_{L_2(Q)} \nu (dz)
$$
\begin{equation}                     \label{monotonicity}
\leq \varepsilon|u-v|^2_{H^1_0(Q)}+N(\varepsilon)|u-v|^2_{L_2(Q)},
\end{equation}
$$
2(f_t(u,\nabla u)-f_t(v,\nabla v), (u-v)^+)+\int_Z|I_{u>v}(g_t(z,u)-g_t(z,v))|^2_{L_2(Q)} \nu(dz)
$$
\begin{equation}                     \label{monotonicity with +}
\leq \varepsilon|(u-v)^+|^2_{H^1_0(Q)}+N(\varepsilon)|(u-v)^+|^2_{L_2(Q)},
\end{equation}
where $N(\varepsilon)$ depends only on $\varepsilon$ and $K$.
\end{lemma}
\begin{proof}
We show \eqref{monotonicity with +}. Using the second part of  Assumption \ref{monotnicity of f} and Young's inequality we have
$$
2(f_t(u,\nabla u)-f_t(v,\nabla v), (u-v)^+) \leq  \frac{K}{\varepsilon}|(u-v)^+|_{L_2(Q)}^2+\varepsilon |\nabla (u-v)^+|_{L_2(Q)}^2
$$ 
$$
+\int_Q  (f_t(x,u, \nabla u)-f_t(x,v, \nabla u))(u(x)-v(x))^+ dx.
$$
This combined with  Assumption \ref{monotnicity of f} gives \eqref{monotonicity with +}. Inequality \eqref{monotonicity} can be shown in the same  way.
\end{proof}

\section{Proof of Theorems \ref{existence equation 1}, \ref{Maintheorem}, \ref{existence equation 2} and \ref{Maintheorem2}}
 We are now ready to proceed with the proofs of the main theorems.
 
 \begin{proof}
 [Proof of Theorems  \ref{existence equation 1} and \ref{existence equation 2}]
 We prove Theorem \ref{existence equation 1}. It suffices to show that conditions I) through IV) from \cite{G3} are satisfied, and then the result follows immediately from Theorems 2.9 and 2.10 of the same article. The growth condition of the operator $L_t+f_t(\cdot)$ can be verified easily.  Notice that for every $\omega, t$ and  $x$, the function $f_t(x,r,r')$ is continuous in $(r,r')$. Using this, ii) from Assumption \ref{Assumptionforspde} and the fact that $L_t$ is a bounded linear operator from $H^1_0(Q)$ into $H^{-1}(Q)$, we see that $L_t+f_t(\cdot)$ is semicontinuous (in the sense of \cite{G3}).  Now, by ii) and iv) from Assumption \ref{Assumptionforspde}, the boundedness of $\phi$ and  \eqref{equ estimate u} we see that for a $\theta >0$ and a constant $C$ we have 
 $$
 2\<L_tu,u\>+2(u,f_t(u,\nabla u))+\sum_k|G^k_t(u)|^2_{L_2(Q)}
 + \int_Z |g_t(u)|^2_{L_2(Q)} \nu(dz)
 $$
 $$
 \leq -\theta |u|^2_{H^1_0(Q)} + C|u|^2_{L_2(Q)}+C|\bar h_t|^2_{L_2(Q)}.
 $$
 for all $t$, $\omega$ and $u \in H^1_0(Q)$. 
This shows that the coercivity condition is satisfied. Using i), iv),  v) from Assumption \ref{Assumptionforspde} and  \eqref{equ estimate u} we see that for all $(t,\omega)$ and $\gamma>0$
 $$
 2\<L_tu-L_tv,u-v\> +\sum_k|G^k(u)-G^k(v)|^2_{L_2(Q)}
 $$
 $$
 \leq (\gamma - \varkappa) |u-v|^2_{H^1_0(Q)} + C(\gamma)|u-v|^2_{L_2(Q)},
 $$ 
 for all $u,v \in H^1_0(Q)$, where $\varkappa$ is the ellipticity constant 
 form part (iv) of Assumption \ref{Assumptionforspde}. 
 Combining this with \eqref{monotonicity} we have that the monotonicity condition 
 is also satisfied. The proof of Theorem \ref{existence equation 2} goes in the same way.  
 We omit the details, we only note that  one  also has to use  
 \eqref{estimate integrals equation 2 without +}.
\end{proof} 

\begin{proof}[Proof of Theorem \ref{Maintheorem}] Without loss of generality we can assume 
that Assumption \ref{monotnicity of f} is satisfied by $f$. For the difference $h=u-v$ we have
$$
h_t=h_0+\int_{(0,t]} L_sh_s+f_s(u_s, \nabla u_s)-F_s(v_s, \nabla v_s) \, ds
$$
$$
+ \int_{(0,t]} \phi^{ki}_s D_ih_s+\sigma^k_s(u_s)-\sigma^k_s(v_s)\,dw^k_s
$$
\[
+\int_{(0,t]} \int_Z g(s,z,u_{s-})-g(s,z,v_{s-}))\tilde{N}(ds,dz).\]
By Theorem \ref{Ito's formula} we have
$$
|h^+_t|^2_{L^2} = \int_{(0,t]} A^{(1)}_s+ A^{(2)}_s+A^{(3)}_s+2\<\mathcal{I}^{(1)}_sh_s,h^+_s \> \ ds+m_t
$$
for a local martingale $m_t$, where
\begin{align}
\nonumber
&A^{(1)}_s=\int_{Q}\Big\{ -2a^{ij}_s(x)D_ih^+_s(x)D_jh^+_s(x) \\  \label{A1}
& +\sum_k \Big| I_{h_s> 0}\sum_i \phi^{ki}_s(x)D_ih_s(x)+I_{h_s> 0}(\sigma^k_s(x,u_s(x))-\sigma^k_s(x,v_s(x)))\Big|^2\Big\} dx \\ \label{A2}
&A^{(2)}_s= 2\int_Q (f_s(x,u_s, \nabla u_s)-F_s(x,v_s, \nabla v_s))h^+_s(x)  dx\\ \nonumber
& A^{(3)}_s=\int_Z \int_Q \{[h_s(x)+g_s(x,z,u_{s-}(x))-g_s(x,z,v_{s-}(x))]^+\}^2 
-|h_s(x)^+|^2\\ \nonumber
&-2 h^+_s(x)[g_s(x,z,u_s(x))-g_s(x,z,v_s(x))] dx \nu(dz).
\end{align}
One can easily see that for every $\varepsilon>0$, there exist $C(\varepsilon)>0$ depending only on $\varepsilon$, $K$ and $d$, such that
$$
A^{(1)}_s \leq (-\varkappa+\varepsilon) |h^+_s|^2_{H^1_0(Q)} +C(\varepsilon) |h^+_s|^2_{L_2(Q)}.
$$
By  Assumption \ref{assumptionforcomparison} 
we obtain 
$$
A^{(3)}_s=  \int_Z \int_Q I_{h_s > 0}|g_s(x,z,u_s)-g_s(x,z,v_s)|^2dx \nu(dz).
$$
Hence, by  \eqref{monotonicity with +} we have
$$
A^{(2)}_s+A^{(3)}_s \leq \varepsilon|h^+_s|^2_{H^1_0(Q)} +C(\varepsilon) |h^+_s|^2_{L_2(Q)}.
$$
 Combining these estimates and using \eqref{equ estimate positive part}
 we have  a constant $C$ such that, almost surely
 $$
|h^+_t|^2_{L^2(Q)} 
\leq C \int_{(0,t]} |h^+_s|^2_{L^2(Q)}\, ds+m_t
\quad\text{ for all $t\in[0,T]$}.
$$
Let $(\tau_n)_{n\in \mathbb{N}}$ be stopping times such that 
$\int_0^{t\wedge \tau_n} |h^+_s|^2_{L^2(Q)}\, ds \leq n$ and almost surely, $\tau_n =T$ for $n$ large enough.
By a standard localization argument and Fatou's lemma we get 
\[
EI_{t \leq \tau_n} |h^+_t|^2_{L^2(Q)} 
\leq C \int_{(0,t]} EI_{s \leq \tau_n}|h^+_s|^2_{L^2(Q)}\, ds< \infty
\quad\text{for all $t\in[0,T]$},
\]
and the result follows by Gronwall's and Fatou's lemmas.
\end{proof}
\begin{proof}[Proof of Theorem \ref{Maintheorem2}]We assume again
that Assumption \ref{monotnicity of f} is satisfied by $f$. For the difference $h=u-v$ we have
$$
h_t=h_0+\int_{(0,t]} \{ \mathcal{L}_sh_s+f_s(u_s, \nabla u_s)-F_s(v_s, \nabla v_s) \} \ ds
$$
$$
+ \int_{(0,t]} \{\phi^{ki}_s D_ih_s+\sigma^k_s(u_s)-\sigma^k_s(v_s)\} \,dw^k_s
 +\int_{(0,t]}  \int_F S_{s,\zeta} h_{s-} \tilde{M}(ds,d\zeta)
$$
By Theorem \ref{Ito's formula} we have
$$
|h^+_t|^2_{L^2(\bR^d)} = \int_{(0,t]} A^{(1)}_s+ A^{(2)}_s+\tilde{\rho}_s(h_s)+\<\mathcal{I}^{(1)}_sh_s,h^+_s \> \ ds+m_t
$$
for a local martingale $m_t$. Here  $A^{(1)}, A^{(2)}$ are as in \eqref{A1}, \eqref{A2} (with the  integration  over $\bR^d$ instead of $Q$), and $\tilde{\rho}$ is defined in \eqref{rho}.
By using the same arguments as in the previous proof, this time also using \eqref{estimate integrals equation 2}, we bring the proof to an end.
\end{proof}




\end{document}